\documentclass[12pt]{amsart}

\RequirePackage{amsmath}
\RequirePackage{amssymb}
\RequirePackage{amsthm}
\RequirePackage{bm}

\usepackage[margin=2cm]{geometry}

\newcommand{\C}{\ensuremath{\mathbb{C}}}

\newcommand{\R}{\ensuremath{\mathbb{R}}}

\newcommand{\Z}{\ensuremath{\mathbb{Z}}}

\newcommand{\abs}[1]{\ensuremath{\lvert#1\rvert}}

\newcommand{\abss}[1]{\ensuremath{\left\lvert#1\right\rvert}}

\newcommand{\eps}{\epsilon}

\newtheorem{theorem}{Theorem}[section]
\newtheorem{lemma}[theorem]{Lemma}
\newtheorem{proposition}[theorem]{Proposition}
\newtheorem{corollary}[theorem]{Corollary}

\theoremstyle{definition}
\newtheorem{definition}[theorem]{Definition}

\DeclareMathOperator{\codim}{codim}
\DeclareMathOperator{\supp}{supp}
\DeclareMathOperator{\spn}{span}

\DeclareMathOperator{\Tr}{Tr}
\DeclareMathOperator{\Sym}{Sym}

\DeclareMathOperator{\Spec}{Spec} 

\title{Singularity of Random Matrices over Finite Fields}
\author{Kenneth Maples}
\address{Institut f\"ur Mathematik\\ Universit\"at Z\"urich\\ Winterthurerstrasse 190\\ 8057-Z\"urich, Switzerland}
\email{kenneth.maples@math.uzh.ch}

\begin{document}

\begin{abstract}
Let $A$ be an $n \times n$ random matrix with iid entries over a finite field of order $q$.  Suppose that the entries do not take values in any additive coset of the field with probability greater than $1 - \alpha$ for some fixed $0 < \alpha < 1$.  We show that the singularity probability converges to the uniform limit with an exponentially small error depending only on $\alpha$. We also show that the distribution of the determinant of $A$ converges to its limiting distribution at an exponential rate.
\end{abstract}

\thanks{The author was supported by a Graduate Research Fellowship from the National Science Foundation. This article was developed from part of the author's PhD thesis \cite{PHD}}

\maketitle

\section{Introduction}

Let $q = p^f$ be a prime power and let $\mathbb{F}_q$ be a finite field with $q$ elements.  Suppose $\xi$ is a random variable that takes values in $\mathbb{F}_q$ with probability distribution $\mu$. We say that $\mu$ is \emph{$\alpha$-dense} for $0 < \alpha < 1$ if for every additive subgroup $T \leq \mathbb{F}_q$ and $s \in \mathbb{F}_q$,
    \[
      \mathbb{P}(\xi \in s + T) \leq 1 - \alpha.
    \]
Let $A$ be an $n \times n$ random matrix whose entries are iid copies of $\xi$.  It is a classical problem to compute the typical spectral properties of $A$. The most simple question is to determine how often the matrix is singular.   As it turns out, as long as the distribution of $A$ is not sparse in the sense that its entries take values from an $\alpha$-dense distribution, the probability of singularity converges rapidly to the expected value. The purpose of this article is to give an effective bound for this rate.
\begin{theorem} \label{thm:main}
    Let $\mathbb{F}_{q}$ with $q = p^j$ and suppose $A \in \text{M}(n,\mathbb{F}_q)$ is a random matrix with iid entries which take values from an $\alpha$-dense probability distribution.  Then we have the estimate
    \[
      \mathbb{P}(A \text{ is non-singular}) = \prod_{k=1}^\infty (1 - q^{-k}) + O(e^{-c \alpha n})
    \]
where the implied constant and $c > 0$ are absolute.
\end{theorem}

The first result in this direction is due to Charlap, Rees, and Robbins \cite{CRR90}, who showed that the singularity probability $\mathbb{P}(A \text{ is singular})$ converges to the limiting value,
\[
  \prod_{k=1}^n (1 - q^{-k}) = \frac{\abs{\text{GL}(n,\mathbb{F}_q)}}{\abs{\text{M}(n,\mathbb{F}_q)}}
\]
which is the density of invertible matrices in the set of all $n \times n$ matrices over $\mathbb{F}_q$. Their method was based on M\"obius inversion and did not give an effective rate on the convergence.

The dependence on $\alpha$ in Theorem~\ref{thm:main} is optimum. In fact, if we take $\mu$ such that $\mathbb{P}(\xi = 0) > 1 - c \frac{\log n}{n}$ for $c$ suitably small, then we expect $A$ to have a positive proportion of zero columns.

The techniques used to prove Theorem~\ref{thm:main} also naturally allow us to control the distribution of the determinant. As an example, we have the following result.
\begin{theorem} \label{thm:probdistribution}
    For all non-zero $t \in \mathbb{F}_q$, we have the formula
\[
  \mathbb{P}(\det A = t) = q^{-1} \prod_{k = 2}^\infty (1 - q^{-k}) + O(e^{-c \alpha n}).
\]
where $c > 0$ and the implied constant are absolute.
\end{theorem}

The analysis of the singularity of random matrices with continuous distributions over $\C$ is trivial. Indeed, if the last $n-1$ columns of an $n \times n$ matrix are linearly independent, they span a hypersurface in $\C^n$.  Hypersurfaces have Lebesgue measure zero, so as long as the probability distribution of $X$ is absolutely continuous the matrix is almost surely non-singular. In contrast, if the law of the entries takes any value with a positive probability, then the matrix can be singular with positive probability. For example, if the distribution assumes the value $1$ with probability $\delta > 0$, then the probability that the first two columns are identical (and thus the matrix is singular) is bounded below by $\delta^{2n}$.

The first upper bound on the probability that an iid random matrix over $\C$ with non-continuous distribution is singular was given by Kom\'os \cite{Kom67, Kom68}, first for the Bernoulli distribution and later generalized to other laws. An analysis of the argument shows that he in fact gave an upper bound of $O(n^{-\frac12})$ for the probability that the matrix is singular. This was improved to an exponential rate by Kahn, Koml\'os, and Szemer\'edi in \cite{KKS95}, who gave an exponential rate $O(e^{-c n})$ using a hypergraph argument. Subsequent work by Tao and Vu \cite{TV06b, TV07} and Bourgain, Vu, and Wood \cite{BVW10} have improved estimates for the exponential rates close to the conjectured value.

Progress on this problem has come from new estimates for the Littlewood-Offord problem. The classical Littlewood-Offord problem asks, given a fixed vector $a \in \R^n$, for the proportion of signed sums $\pm a_1 \pm \cdots \pm a_n$ lying in an interval in $\R$. This question first arose in the work of Littlewood and Offord \cite{LO43} on the real zeros of a random polynomial. Their estimate, based on a dyadic pigeonhole principle, was improved by Erd\H{o}s \cite{Erd45} to a sharp bound. Later progress on this problem showed that there is a correspondence between upper bounds for this quantity and additive structure among the coefficients of the vector $a$.

In this paper we prove three Littlewood-Offord theorems for finite fields.  We have replaced the random sums from the classical problem with sums $w \cdot X$, where $w \in \mathbb{F}_q^n$ is a fixed vector and $X$ is random with iid entries taken from an $\alpha$-dense probability distribution.  The techniques we employ go back to Hal\'asz in \cite{Hal77}, who bounded the probability $\mathbb{P}(X \cdot w = 0)$  by finding additive structure in the level sets of the Fourier transform of $1_{X \cdot w = 0}$.  We also crucially rely on arguments developed in \cite{KKS95} and \cite{TV06b}.

The key new advance to study matrices over finite fields is an inverse theorem for random sums $w \cdot X$ which are almost uniformly distributed, but differ from the uniform distribution by an exponentially small quantity. In this setting we show that the coefficients of $w$ must lie in a small subset $R \subseteq \mathbb{F}_q^n$. Because the sums are almost uniformly distributed, it is possible to enumerate all such vectors.

The requirement that $\mu$ be an $\alpha$-dense probability distribution can be weakened. As discussed in \cite{CRR90}, let $\theta \in \mathbb{F}_8$ be a primitive element and consider the uniform distribution $\mu$ on $\{0, 1, \theta, 1+\theta\}$. With the methods from this article is easy to show that we recover the expected singularity probability $\prod_{k=1}^\infty (1 - 8^{-k}) + O(e^{-cn})$ while $\mu$ is supported on an additive subgroup of $\mathbb{F}_2^3$. This was first done (but without the exponential rate) by Kahn and Koml\'os in \cite{KK01}, where they showed that it suffices to assume that $\mu$ does not concentrate on affine \emph{subfields}; i.e. subsets of the form $\beta \mathbb{F}_{p^d} + \gamma$ for $d \mid f$ and $\beta, \gamma \in \mathbb{F}_q$.

However, with this weaker condition we can construct examples that do not have an exponentially small error term. For example, let $f$ be a large prime, $\theta \in \mathbb{F}_{p^f}$ a primitive element and $\mu$ uniformly distributed on $\{0,1,\theta,1 + \theta\}$. After expanding the determinant we see that $\det A$ can only take values in the additive subgroup $\langle 1, \theta, \theta^2, ..., \theta^n \rangle$.

\section{Reduction of Theorem~\ref{thm:main} to a universality statement} \label{sec:maintheorem}

Let $X_1, ..., X_n$ denote the columns of $A$.  For convenience we will let $X \in \mathbb{F}_q^n$ denote an independent random vector with iid entries distributed according to $\mu$; thus each $X_\ell$ is an iid copy of $X$ for $1 \leq \ell \leq n$.

We expose each column $X_k$ in turn, from $X_n$ to $X_1$, and check whether it lies in the span of the previously exposed columns.  Let $W_k := \langle X_{k+1}, ..., X_n \rangle$ denote the span of the final $n - k$ columns of $A$. By conditional expectation,
\[
  \mathbb{P}(A \text{ is non-singular}) = \prod_{k=1}^n \mathbb{P}(X_k \notin W_k \mid \codim W_k = k)
\]

Suppose that $V$ is a deterministic subspace of $\mathbb{F}_q^n$ of codimension $k$.  If $X_k$ were chosen \emph{uniformly} from $\mathbb{F}_q^n$ then we would have $\mathbb{P}(X_k \notin V) = 1 - q^{-k}$ and the theorem would follow.  In fact, we can show that for sufficiently small $k$ this equality holds with exponentially small error.
\begin{proposition}\label{prop:uniformsubspaces}
    There is an absolute constant $\eta > 0$ such that, for all $1 \leq k \leq \eta n$, we have the estimate
\[
  \mathbb{P}(X_k \in W_k \mid \codim W_k = k) = q^{-k} + O(e^{-c\alpha n})
\]
where the implied constant and $c > 0$ are absolute.
\end{proposition}
For columns $X_k$ with $k > \eta n$ we have $q^{-k} = O(e^{-cn})$, so it suffices to show that
\[
  \mathbb{P}(X_k \in W_k \mid \codim W_k = k) = O(e^{-c\alpha n}).
\]
This is guaranteed by the following lemma, first recorded in \cite{Odl88}.
\begin{lemma}[Odlyzko] \label{lem:odlyzko}
For any fixed subspace $V$ of $\mathbb{F}_q^n$ and random vector $X \in \mathbb{F}_q^n$ that is $\alpha$-dense, we have the bound
\[
  \mathbb{P}(X \in V) \leq (1-\alpha)^{\codim V}.
\]
\end{lemma}
\begin{proof}[Proof of Lemma~\ref{lem:odlyzko}]
Let $k$ denote the codimension of $V$.  We can find $n - k$ coordinates $\tau \subseteq [n]$ such that $V$ is a graph over $\tau$.  If we condition on the coordinates of $X$ in $\tau$, then there is a unique choice for the remaining coordinates $[n] \setminus \tau$ for $X \in V$.  Since $\mu$ is $\alpha$-dense, the probability that each entry of $X$ assumes the required value is bounded by $1 - \alpha$, and the result follows from the independence of the entries.
\end{proof}

We will now prove Proposition~\ref{prop:uniformsubspaces}.  It is convenient to distinguish four kinds of subspaces that $W_k$ can represent as $X_{k+1}, ..., X_n$ vary. Fix absolute constants $\delta$, $d$, and $D$; for intuition we can take $\delta = 1/100$, $d = 1/100$, and $D = 10$, but we do not compute exact values.

Let $V$ be a \emph{fixed} codimension $k$ subspace of $\mathbb{F}_q^n$. Then we say that $V$ is \emph{sparse}, \emph{unsaturated}, \emph{semi-saturated}, or \emph{saturated} as follows.
\begin{description}
    \item[sparse]  There is a non-zero $w \perp V$ with $\abs{\supp w} \leq \delta n$.  We can directly count these subspaces.
    \item[unsaturated] $V$ is not sparse and we have the estimate 
    \[\max(e^{-d\alpha n}, D q^{-k}) < \abs{\mathbb{P}(X \in V) - q^{-k}}.\]
    We adapt the swapping method from \cite{TV06b} and construct a random vector $Y$ such that $\mathbb{P}(X \in V) \leq (\frac12 + \frac1D + o(1)) \mathbb{P}(Y \in V)$.
    \item[semi-saturated] $V$ is not sparse and we have the estimates 
    \[e^{-d\alpha n} < \abs{\mathbb{P}(X_k \in V) - q^{-k}} \leq D q^{-k}.\]
    In this range the swapping method does not yield a useful gain; however, we can enumerate semi-saturated $V$ by finding a structured $w \perp V$. Note that for $q$  sufficiently large there are no semi-saturated spaces.
    \item[saturated] $V$ is not sparse and we have the estimate 
    \[\abs{\mathbb{P}(X_k \in V) - q^{-k}} \leq e^{-d\alpha n}.\]
\end{description}

Proposition~\ref{prop:uniformsubspaces} will follow if we can show that $W_k$ represents a saturated subspace with probability $1 - O(e^{-c\alpha n})$ with absolute constants.  It therefore suffices to show that $W_k$ is sparse, semi-saturated, or unsaturated with probability $O(e^{-c\alpha n})$.

\subsection{Sparse subspaces} \label{sec:sparsesubspaces}

We adapt the counting method from \cite{KKS95}.  If $W_k$ is sparse, then we can find a non-zero $w \perp W_k$ with $\abs{\supp w} \leq \delta n$.  By the union bound,
\[
  \mathbb{P}(W_k \text{ is sparse}) \leq \sum_{\substack{\sigma \subseteq[n] \\ 1 \leq \abs{\sigma} \leq \delta n}} \mathbb{P}(W_k \perp w \text{ for some } w \text{ with } \supp w = \sigma).
\]
Fix $\sigma$. It suffices to bound
\[
  Q_\sigma := \mathbb{P}(W_k \perp w \text{ for some } w \text{ with } \supp w = \sigma) \leq O(e^{-c\alpha n})
\]
with the implied constant and $c > 0$ depending on $\delta$. We will choose $\delta$ in the proofs for unsaturated and semi-saturated subspaces.

If we have such a perpendicular vector $w$ we can write the matrix equation
\[
  w^t \begin{bmatrix} X_{\ell + 1} & \cdots & X_n \end{bmatrix} = 0.
\]
Restricting the product to indices in $\sigma$ and denoting this reduction by $\widetilde{\cdot}$,
\[
  \widetilde{w}^t \begin{bmatrix} \widetilde{X}_{\ell + 1} & \cdots & \widetilde{X}_n \end{bmatrix} = 0.
\]
The matrix of reduced columns has size $\abs{\sigma} \times (n - \ell)$ and has rank less than $\abs{\sigma}$, so we conclude that the dimension of the column space is at most $\abs{\sigma}-1$.  There are at most $\binom{n-\ell}{\abs{\sigma}-1}$ possible choices for a set $\tau$ of spanning columns; we do not require that they be linearly independent.  Regardless of the choice of $\tau$, the remaining columns must be perpendicular to $\widetilde{w}$.  Collecting these bounds, we find
\[
  Q_\sigma \leq \sum_{\substack{\tau \subset [n] \\ \abs{\tau} = \abs{\sigma} - 1}}  \sup_{\supp w = \sigma} \mathbb{P}(X_t \perp w \text{ for all } t \notin \tau \mid \codim W_k = k)
\]
We expect linearly independent vectors to be less likely to lie in a given subspace than average.  The next proposition verifies that intuition.
\begin{proposition} \label{prop:independencebound}
    Let $Z_1, ..., Z_r$ be non-trivial iid random vectors in $\mathbb{F}_q^n$.  Then we have the bound
    \[
      \mathbb{P}(Z_1, \ldots, Z_r \in V \mid Z_1, \ldots, Z_r \text{ are linearly independent}) \leq \mathbb{P}(Z \in V)^r.
    \]
\end{proposition}
\begin{proof}
Expanding the left hand side with conditional expectation,
\[
  \prod_{j=1}^r \mathbb{P}(Z_j \in V \mid Z_1, \ldots, Z_{j-1} \in V \text{ and } Z_1, ..., Z_j \text{ are linearly independent})
\]
Let $U := \langle Z_1, \ldots, Z_{j-1} \rangle \leq V$ denote the span of the exposed vectors.  It suffices to show that
\[
  \frac{\mathbb{P}(Z \in V \setminus U)}{\mathbb{P}(Z \notin U)} \leq \mathbb{P}(Z \in V).
\]
In fact,
\begin{align*}
    \mathbb{P}(Z \in V \setminus U) &= \mathbb{P}(Z \in U) \mathbb{P}(Z \in V \setminus U) + \mathbb{P}(Z \notin U) \mathbb{P}(Z \in V \setminus U) \\
    &\leq \mathbb{P}(Z \in U) \mathbb{P}(Z \notin U) + \mathbb{P}(Z \in V \setminus U) \mathbb{P}(X \notin U) \\
    &= (\mathbb{P}(Z \in U) + \mathbb{P}(Z \in V \setminus U)) \mathbb{P}(Z \notin U)
\end{align*}
and the proposition follows.
\end{proof}

Combining terms we get
\[
  \mathbb{P}(W_k \text{ is sparse}) \leq \sum_{\substack{\sigma \subseteq[n] \\ 1 \leq \abs{\sigma} \leq \delta n}} \binom{n-\ell}{\abs{\sigma}-1} \sup_{\supp w = \sigma} \mathbb{P}(X \perp w)^{n-\ell-\abs{\sigma}+1}
\]
It remains to bound $\mathbb{P}(X \perp w)$ for $w$ with support $\sigma$. For this task we can use the following Littlewood-Offord theorem.

\begin{lemma}[Littlewood-Offord] \label{lem:littlewoodofford}
    Let $X \in \mathbb{F}_q^n$ be a random vector with iid entries taken from an $\alpha$-dense probability distribution $\mu$. Suppose $w \in \mathbb{F}_q^n$ has at least $m$ non-zero coefficients.  Then we have the estimate
    \[
      \abss{\mathbb{P}(X \cdot w = r) - \frac1q} \lesssim \frac{1}{\sqrt{\alpha m}}
    \]
    for all $r \in \mathbb{F}_q$, where the implied constant is absolute.
\end{lemma}
We only require the estimate for $r = 0$. We will prove Lemma~\ref{lem:littlewoodofford} in Section~\ref{sec:lo}.

If we combine this with the trivial inequality $\mathbb{P}(X \perp w) \leq 1 - \alpha$ for small $\abs{\sigma}$, we deduce
\[
  \mathbb{P}(W \text{ is sparse}) \leq O(e^{-c\alpha n})
\]
with absolute constants for $\delta$ sufficiently small.

\subsection{Semi-Saturated subspaces}

Let $V$ be a semi-saturated subspace of codimension $k$. We first claim that we can find a non-zero $\xi \perp V$ that is structured in the following sense.
\begin{proposition} \label{prop:inversetheorem}
For all $\beta > 0$ there is a value of $d$ in the definition of semi-saturated and a subset
\[
  R \subseteq \mathbb{F}_q^n, \qquad \abs{R} \leq \beta^n q^n
\]
such that every semi-saturated $V$ is perpendicular to a non-zero $\xi \in R$.
\end{proposition}
We will prove Proposition~\ref{prop:inversetheorem} in Section~\ref{sec:lo}.

We have therefore found a $\xi \in V^\perp$ that is ``structured'' in that it lies in an exponentially small subset of $\mathbb{F}_q^n$.  It turns out that this is enough to attain the desired estimate on $\mathbb{P}(W_k \text{ is semi-saturated})$.  In fact, we estimate
\[
  \mathbb{P}(W_k \text{ is semi-saturated} \mid \codim W_k = k) \leq \sum_{\substack{\codim V = k \\ V \text{ is semi-saturated}}} \mathbb{P}(W_k = V | \codim W = k).
\]
Using Proposition~\ref{prop:independencebound}, we can bound
\[
  \mathbb{P}(W_k = V \mid \codim W_k = k) \leq \mathbb{P}(X \in V)^{n-k} \leq D^{n-k} q^{-k(n-k)}
\]
where the last inequality is from the definition of semi-saturated $V$.  It now suffices to count the number of semi-saturated subspaces.

The subspace $V$ is completely determined by its annihilator $V^\perp$.  We therefore count the number of possible annihilators that meet $R$.  We can choose $k$ generators $v_1, \ldots, v_k$ for $V^\perp$ and force $v_1 \in R$; we then divide by the number of ways we could generate the same subspace with different choices for $v_2, \ldots, v_k$.  This gives the upper bound
\[
  \#\{\text{semi-saturated } V\} \lesssim \beta^n q^n \frac{ (q^n)^{k-1}}{\abs{V^\perp}^{k-1}} \leq \beta^n q^{nk - k^2 + k}.
\]
Collecting terms we find
\[
  \mathbb{P}(W_k \text{ is semi-saturated}) \lesssim D^{n-k} \beta^n q^k
\]
If there are any semi-saturated subspaces, we must have the inequality $e^{-d\alpha n} \leq D q^{-k}$. With fixed $D$ we can choose $\beta$ and therefore an upper bound for $d$ such that the right hand side converges to zero at an exponential rate.

\subsection{Unsaturated subspaces} \label{sec:unsaturatedsubspaces}

In \cite{TV06b} it was observed that there is a random vector $Y$ such that if $X \in \R^n$ is a Bernoulli random vector and $V$ is a non-sparse hyperplane, then we can bound
\[
  \mathbb{P}(X \in V) \leq (\frac12 + o(1)) \mathbb{P}(Y \in V).
\]
Ignoring difficulties with independence, this suggests the inequality
\[
  \mathbb{P}(X_{k+1}, ..., X_n \spn V) \leq c^n \mathbb{P}(Y_{k+1}, ..., Y_n \spn V)
\]
for some $1/2 < c < 1$. Summing over non-saturated subspaces $V$ and using the trivial bound
\[
  \sum_{\substack{V \text{ unsaturated} \\ \codim V = k}} \mathbb{P}(Y_{k+1}, ..., Y_n \spn V) \leq 1
\] would complete the argument.

Over the finite field $\mathbb{F}_q$ we cannot quite get the above inequality, but rather an inequality of the form
\[
  \abs{\mathbb{P}(X \in V) - q^{-k}} \leq (\frac12 + o(1)) \abs{\mathbb{P}(Y \in V) - q^{-k}}.
\]
This reflects our intuition that Fourier analysis over $\mathbb{F}_q$ controls errors from univormity rather than absolute probabilities.  If we want to use this inequality to get an exponential strength gain, then we must require $\mathbb{P}(X \in V) - q^{-k} > D q^{-k}$ for some $D > 0$.  It turns out that this is enough for the argument to work.

Let $\nu$ denote a probability distribution to be chosen later.  Suppose $\nu$ is $\beta$-dense for some $\beta > 0$; we will later show that $\beta = \alpha/8$. Let $Y_1, \ldots, Y_r \in \mathbb{F}_q^n$ be iid random vectors with iid entries taken from $\nu$ and let $Z_1, \ldots, Z_s \in \mathbb{F}_q^n$ be iid copies of $X$. Here $r, s$ are parameters to be chosen later.

We will need control over $\mathbb{P}(X \in V)$ in the sequel.  We therefore make the following definition, first given in \cite{TV06b}.
\begin{definition}
    Let $V$ be a deterministic subspace in $\mathbb{F}_q^n$.  We say that $V$ has combinatorial codimension $d_\pm \in \Z^+/n$ and write $d_{\pm}(V) = d_\pm$ if
    \[
      (1-\alpha)^{d_\pm} \leq \mathbb{P}(X \in V) < (1-\alpha)^{d_\pm-1/n}
    \]
\end{definition}
Note that the combinatorial codimension of a subspace depends on the choice of $\alpha$ and $\mu$. There are $O(n^2)$ possible combinatorial codimensions, so it suffices to control each separately.

For the rest of this section we will assume that $X_{k+1}, ..., X_n$ are conditioned to be linearly independent. Fix an unsaturated subspace $V$ with codimension $k$ and combinatorial codimension $d_\pm$.  Let $B_V$ denote the event
\[
  B_V := \{Y_1, \ldots, Y_r, Z_1, \ldots, Z_s\text{ are linearly independent in } V\}.
\]
By probabilistic independence we can write
\[
  \mathbb{P}(W_k = V) = \frac{\mathbb{P}(B_V \wedge W_k = V)}{\mathbb{P}(B_V)}.
\]
If $X_{k+1}, \ldots, X_n \spn V$, we can find $n - k - r - s$ columns that complete $Y_1, \ldots, Y_r, Z_1, \ldots, Z_s$ to a basis for $V$.  The remaining vectors must also lie in $V$.  We therefore define the event
\[
  C_V := \{X_{k+r+s+1}, \ldots, X_n, Y_1, \ldots, Y_r, Z_1, \ldots, Z_s \spn W\}
\]
so after relabeling the columns of $A$,
\[
  \mathbb{P}(B_V \wedge W = V) \leq \binom{n-k}{r+s} \mathbb{P}(X_{k+1}, \ldots, X_{k+r+s} \in V) \mathbb{P}(C_V)
\]
By Proposition~\ref{prop:independencebound}, recalling that our vectors $X_{k+1}, ..., X_n$ are conditioned to be linearly independent,
\[
  \mathbb{P}(X_{k+1}, \ldots, X_{k+r+s} \in V) \leq \mathbb{P}(X \in V)^{r+s}.
\]

Next we consider $\mathbb{P}(B_V)$.  We can write by conditional expectation
\[
  \mathbb{P}(B_V) = \mathbb{P}(B_V \mid Y_1, \ldots, Y_r, Z_1, \ldots, Z_s \in V) \mathbb{P}(Y \in V)^r \mathbb{P}(Z \in V)^s.
\]
We need to control the probability that the vectors $Y_1, ..., Y_r,Z_1, ..., Z_s$ are linearly independent.  It turns out that Odlyzko's lemma is strong enough for what we need, as long as $r$ and $s$ are not too large and the combinatorial codimension is not too small.
\begin{proposition}
Let $Y_1, ..., Y_r$ be iid vectors taken from a $\beta$-dense probability distribution $\nu$ and let $Z_1, ..., Z_s$ be iid vectors taken from an $\alpha$-dense probability distribution $\mu$.  Then if $V$ has combinatorial codimension $d_\pm \leq O_{\alpha,\beta}(n)$ we have
\[
  \mathbb{P}(B_V \mid Y_1, \ldots, Y_r, Z_1, \ldots, Z_s \in V) \geq \frac12.
\]
\end{proposition} 
\begin{proof}
Define the events
\[
  F_V := \{Y_1, ..., Y_r, Z_1, ..., Z_s \in V\}.
\]
and, for convenience,
\begin{align*}
  F_V(i) &:= \{F_V \wedge Y_1, ..., Y_{i-1} \text{ are linearly independent}\} \\
  \widetilde{F}_V(j) &:= \{F_V \wedge Y_1, ..., Y_r, Z_1, ..., Z_{j-1} \text{ are linearly independent.}\}
\end{align*}
Expanding the probability with conditional expectation,
\[
  \mathbb{P}(B_V \mid F_V) = \prod_{i=1}^r \mathbb{P}(Y_i \notin \langle Y_1, \ldots Y_{i-1} \rangle \mid F_V(i)) \prod_{j=1}^s \mathbb{P}(Z_j \notin \langle Y_1, \ldots, Y_r, Z_1, \ldots, Z_{j-1} \rangle \mid\widetilde{F}_V(j)).
\]
With Lemma~\ref{lem:odlyzko},
\[
  \mathbb{P}(Y_i \notin \langle Y_1, \ldots Y_{i-1} \rangle \mid F_V(i)) \geq 1 - (1 - \beta)^{n-i+1} (1-\alpha)^{-d_\pm}
\]
and
\[
  \mathbb{P}(Z_j \notin \langle Y_1, \ldots, Y_r, Z_1, \ldots, Z_{j-1} \rangle \mid \widetilde{F}_V(j)) \geq 1 - (1 - \alpha)^{n-r-j+1} (1-\alpha)^{-d_\pm}
\]
We therefore have the lower bound
\[
  \mathbb{P}(B_V \mid F_V) \geq 1 - (1 - \beta)^{n-i+1} (1 - \alpha)^{-d_\pm} - (1 - \alpha)^{n-r-j+1} (1 - \alpha)^{-d_\pm} \geq 1/2
\]
as long as $d_\pm$ is sufficiently small and $r$,$s$ are sufficiently small.
\end{proof}

Collecting estimates, we have
\[
  \mathbb{P}(W = V) \lesssim \binom{n-k}{r+s} \frac{\mathbb{P}(X \in V)^r}{\mathbb{P}(Y \in V)^r} \mathbb{P}(C_V)
\]

We are now ready to state the key lemma to compare the random vectors $X$ and $Y$.

\begin{lemma}[Swapping] \label{lem:replacement}
There is a $\beta$-dense probability distribution $\nu$ on $\mathbb{F}_q$ with $\beta = \alpha/8$ such that, if $Y \in \mathbb{F}_q^n$ is a random vector with iid coefficients distributed according to $\nu$, then
\[
  \abss{\mathbb{P}(X \in V) - q^{-1}} \leq \left( \frac12 + o(1) \right) \abss{\mathbb{P}(Y \in V) - q^{-1}}.
\]
\end{lemma}
If $V$ is unsaturated, then as an immediately corollary we have
\[
 \mathbb{P}(X \in V) \leq \left( \frac12 + \frac1D + o(1) \right) \mathbb{P}(Y \in V)
\]
We will prove this lemma in Section~\ref{sec:lo}. With this estimate, we can sum over all subspaces of codimension $k$ and combinatorial codimension $e$. Since a set of vectors can span at most one subspace, the events $C_V$ for $V$ varying are disjoint and we can conclude
\[
  \sum_{\substack{V : \codim V = k \\ d_{\pm}(V) = d_\pm}} \mathbb{P}(W = V) \lesssim \binom{n-k}{r+s} 2^{-r} = O(e^{-cn}).
\]

Here we picked $r = \delta_1 n$, $s = n - k - r - \delta_2 n$. \qed


\section{Littlewood-Offord Theorems} \label{sec:lo}

We now come to the heart of the argument: proving the three Littlewood-Offord type lemmas used in the preceding section.

We briefly review some theory from additive combinatorics. For more discussion, see \cite{TV06a}.

The following cosine inequality is elementary.
\begin{lemma} \label{lem:cosines}
For all positive integers $k$ and for any $\beta_1, \ldots, \beta_k \in \R$ we have the inequality
\[
  \cos(\beta_1 + \cdots + \beta_k) \geq k \sum_{\ell = 1}^k \cos \beta_\ell - k^2 + 1.
\]
\end{lemma}

\begin{proof}
We can assume that $-\pi/2 \leq \beta_\ell \leq \pi/2$ for all $\ell$, as otherwise the inequality is trivial.  On this interval $\cos$ is concave, so we have the inequality
\[
  k^{-1} \sum_{\ell = 1}^k \cos \beta_\ell \leq \cos \left( \frac{\beta_1 + \cdots + \beta_k}{k} \right).
\]
It suffices to show that
\[
  \cos (\beta/k) \leq k^{-2} \cos \beta + 1 - k^{-2}
\]
for all $\beta \in \R$, but this is immediate from the power series.
\end{proof}

Let $\mu$ be a probability measure on the finite field $\mathbb{F}_q$.  We need estimates on the Fourier transform
\[
  \widehat{\mu}(\psi) := \sum_{t \in \mathbb{F}_q} \mu(t) \psi(t).
\]
Recall that $\mathbb{F}_q \cong \widehat{\mathbb{F}}_q$ via the isomorphism that sends $t \in \mathbb{F}_q$ to the character $x \mapsto e_p(\Tr(tx))$, where $\Tr : F_{p^f} \to F_p$ is the field trace. We define the additive spectrum $\Spec_{1-\eps}\mu$ to be the set
    \[
      \Spec_{1-\eps}\mu := \{ \psi \in \widehat{\mathbb{F}}_q \mid \abs{\widehat{\mu}(\psi)} \geq 1 - \eps \}.
    \]
For $\eps$ small, we can find additive structure in $\Spec_{1-\eps} \mu$.  The next lemma makes this explicit; see Lemma~4.37 in \cite{TV06a}.
\begin{lemma} \label{lem:spec}
    For $\eps_1, \ldots, \eps_k < 1$ we have the sum-set inclusion
    \[
      \Spec_{1 - \eps_1} \mu + \cdots + \Spec_{1 - \eps_k} \mu \subseteq \Spec_{1 - k (\eps_1 + \cdots + \eps_k)} \mu.
    \]
\end{lemma}

\begin{proof}
Let $\psi_\ell \in \Spec_{1-\eps_\ell} \mu$ for each $\ell$.  We write $\psi_\ell(t) = e(\Tr(s_\ell t)/p)$ for appropriate $s_\ell$.  We can find $\theta_\ell \in \R/\Z$ so that
\[
  \text{Re} \sum_{t \in \mathbb{F}_q} \mu(t) e(\Tr(s_\ell t)/p + \theta_\ell) \geq 1 - \eps_\ell.
\]
Summing, we derive
\[
  \text{Re} \sum_{t \in \mathbb{F}_q} \mu(t) \left( k e(\Tr(s_1 t)/p + \theta_1) + \cdots + k e(\Tr(s_k t)/p + \theta_k) - k^2 + 1 \right)
  \geq 1 - k(\eps_1 + \cdots + \eps_k).
\]
The result now follows from Lemma~\ref{lem:cosines}.
\end{proof}

A subset $A \subseteq Z$ of an abelian group induces a symmetry subgroup of $Z$ given by
\[
  \Sym A := \{ h \in Z \mid h + A = A \}
\]
Clearly $A$ can be decomposed into the union of cosets of $\Sym A$.  

We need to bound sumsets from below.  For $q = p$ the estimate we need is the Cauchy-Davenport inequality: any $A, B \subseteq \mathbb{F}_p$ satisfy $\abs{A + B} \geq \min(\abs{A} + \abs{B} - 1, p)$. The next lemma generalizes the Cauchy-Davenport inequality to non-cyclic groups; see Theorem~5.5 in \cite{TV06a} for a proof.
\begin{lemma}[Kneser's Theorem]
Let $A, B \subseteq Z$ be finite subsets of an abelian group $Z$. We have the lower bound
\[
  \abs{A + B} + \abs{\Sym(A + B)} \geq \abs{A} + \abs{B}.
\]
\end{lemma}
Since $\Sym(A_1 + \cdots + A_k)$ is increasing in $k$, we get the following iterated version.
\begin{corollary} \label{cor:iteratedkneser}
Let $A_1, ..., A_k \subseteq Z$ be finite subsets of an abelian group $Z$. We have the lower bound
\[
  \abs{A_1 + \cdots + A_k} + (k-1) \abs{\Sym(A_1 + \cdots + A_k)} \geq \abs{A_1} + \cdots + \abs{A_k}.
\]
\end{corollary}

\subsection{The Classical Littlewood-Offord Estimate}

We start by bounding the concentration probability $\mathbb{P}(X \cdot w = r)$ for arbitrary $r \in \mathbb{F}_q$, $X \in \mathbb{F}_q^n$ a random vector with iid entries taken from an $\alpha$-dense probability measure $\mu$, and $w \in \mathbb{F}_q^n$ a vector with at least $m$ non-zero entries.

\begin{proof}[Proof of Lemma~\ref{lem:littlewoodofford}]
    Let $\xi_1, \ldots, \xi_n$ denote the entries of $X$. We can decompose the concentration probability into its Fourier transform,
\[
  \mathbb{P}(X \cdot w = r) = q^{-1} + q^{-1} \sum_{t \in \mathbb{F}_q \setminus \{0\}} e_p(\Tr (-rt))\prod_{\ell = 1}^n \mathbb{E} e_p(\Tr (\xi_\ell w_\ell t)).
\]
By the triangle inequality,
\[
  \abss{\mathbb{P}(X \cdot w = r) - q^{-1}} \leq q^{-1} \sum_{t \in \mathbb{F}_q \setminus \{0\}} \prod_{\ell = 1}^n \abs{\mathbb{E} e_p(\Tr (\xi_\ell w_\ell t))}
\]
Note that $\mathbb{E}e_p(\Tr (\xi_\ell w_\ell t)) = \widehat{\mu}(w_\ell t)$.

We define $\psi(t) := 1 - \abs{\widehat{\mu}(t)}^2$ so that, with the inequality $\abs{\theta} \leq \exp( - \frac12 (1 - \theta^2))$, we have
\[
  \abss{\mathbb{P}(X \cdot w = r) - q^{-1}} \leq q^{-1} \sum_{t \in \mathbb{F}_q \setminus \{0\}} \exp\left(- \frac12 \sum_{\ell = 1}^n \psi(w_\ell t) \right)
\]
Put $f(t) := \sum_\ell \psi(w_\ell t)$.  We can decompose the sum into level sets,
\[
  \abss{\mathbb{P}(X \cdot w = r) - q^{-1}} \leq \frac12 \int_0^\infty q^{-1} \abs{\{t \neq 0 \mid f(t) \leq v\}} e^{-v/2} \, dv.
\]
Let $T(v) := \{t \mid f(t) \leq v\}$ and $T'(v) := T(v) \setminus \{0\}$.  

We claim the following sum-set inequality; see \cite{Hal77} for the torsion-free case.
\begin{proposition}
    For any $v > 0$, we have the inclusion
    \[
      T(v) + \cdots + T(v) \subseteq T(k^2 v)
    \]
    where there are $k$ terms in the sum.
\end{proposition}

\begin{proof}
We first observe that for any $\beta_1$, \ldots, $\beta_k \in \mathbb{F}_q$, we have the inequality
\[
  \psi(\beta_1 + \cdots + \beta_k) \leq k(\psi(\beta_1) + \cdots + \psi(\beta_k)).
\]
In fact, we can rewrite this equation as
\begin{multline*}
  1 - \sum_{a, b \in \mathbb{F}_q} \mu(a) \mu(-b) \cos(\frac{2\pi}{p} \Tr((a + b)(\beta_1 + \cdots + \beta_k))) \\
  \leq k^2 - k \sum_{j=1}^k \sum_{a,b \in \mathbb{F}_q} \mu(a) \mu(-b) \cos(\frac{2\pi}{p} \Tr((a + b) \beta_j))
\end{multline*}
which follows from Lemma~\ref{lem:cosines}.

Suppose $t_1, \ldots, t_k$ satisfy $f(t_k) \leq v$.  Then we have
\[
  f(t_1 + \cdots + t_k) = \sum_{\ell = 1}^n \psi(w_\ell t_1 + \cdots + w_\ell t_k) \leq k \sum_{j=1}^k \sum_{\ell = 1}^n \psi(w_\ell t_j) \leq k^2 v
\]
as required.
\end{proof}

By Corollary~\ref{cor:iteratedkneser} we deduce
\[
  k \abs{T(v)} \leq \abs{T(k^2 v)} + (k-1) \abs{\Sym(T(v) + \cdots + T(v))}.
\]
This inequality is effective as long as $\abs{\Sym(T(v) + \cdots + T(v))} = 1$.  If $\Sym(T(v) + \cdots + T(v)) \neq \{0\}$, then because $T(v) + \cdots + T(v) \subseteq T(k^2 v)$ we can find a non-trivial additive subgroup $H \leq \mathbb{F}_q$ contained in the set $T(k^2 v)$.  It therefore suffices to choose $k$ such that $T(k^2 v)$ contains no non-trivial additive subgroups.

Fix $H$; we will find a $t \in H$ where $f$ is large. Averaging $f$ over the subgroup,
\[
  \abs{H}^{-1} \sum_{t \in H} f(t) = \sum_{\ell = 1}^n \abs{H}^{-1} \sum_{t \in H} \psi(w_\ell t) = \sum_{\ell = 1}^n \abs{H}^{-1} \sum_{t \in H} (1 - \abs{\widehat{\mu}(w_\ell t)}^2).
\]
By the inverse Fourier transform and the $\alpha$-density of $\mu$,
\[
  \abs{H}^{-1} \sum_{t \in H} \abs{\widehat{\mu}(w_\ell t)}^2 = \sum_{\xi, \zeta \in \mathbb{F}_q} \mu(\xi) \mu(\zeta) 1_{H^\perp}(w_\ell (\xi - \zeta)) \leq 1 - \alpha
\]
Since at least $m$ of the coefficients $w_\ell$ are non-zero,
\[
  \abs{H}^{-1} \sum_{t \in H} f(t) \geq \alpha m.
\]
By the pigeonhole principle, there must be a $t \in H$ with $f(t) \geq \alpha m$.

We therefore conclude that
\[
  \abs{T'(v)} \lesssim \sqrt{\frac{v}{\alpha m}} \abs{T'(\alpha m)}
\]
for all $v \leq \alpha m$.  Inserting this inequality into the level set estimate gives the bound
\[
  \abss{\mathbb{P}(X \cdot w \equiv r) - q^{-1}} \lesssim \frac{1}{\sqrt{\alpha m}} \int_0^\infty \sqrt{v} e^{-v} \, dv + e^{-\alpha m / 2}
\]
as required.
\end{proof}

\subsection{The Inverse Theorem}

We can find our structured perpendicular vector $\xi$ with the pigeonhole principle.
\begin{proof}[Proof of Proposition~\ref{prop:inversetheorem}]
Let $k := \codim V$. We take Fourier transforms to find
\[
  \abss{\mathbb{P}(X \in V) - q^{-k}} \leq q^{-k} \sum_{\zeta \in V^\perp \setminus \{0\}} \prod_{\ell = 1}^n \abs{\widehat{\mu}(\zeta_\ell)}
\]
By the pigeonhole principle, we can bound this above by
\[
  \prod_{\ell = 1}^n \abs{\widehat{\mu}(\xi_\ell)}
\]
for some fixed $\xi \in V^\perp \setminus \{0\}$. Since $V$ is not sparse, $\abs{\supp \xi} \geq \delta n$.

Because $V$ is semi-saturated, we get the lower bound
\[
  e^{-dn} \leq \prod_{\ell = 1}^n \abs{\widehat{\mu}(\xi_\ell)}.
\]
With the estimate $\abs{\theta} \leq \exp( - \frac12 (1 - \theta^2))$ we can take logarithms to find
\[
  \sum_{\ell = 1}^n 1 - \abs{\widehat{\mu}(\xi_\ell)}^2 \leq d n.
\]

Let $\eps = 5 d$.  We can choose $\sigma \subseteq [n]$ with $\abs{\sigma} \geq 0.9 n$ such that
\[
  \xi_\ell \in \Spec_{1 - \eps} \mu
\]
for all $\ell \in \sigma$.

    It suffices to find an absolute $\eta > 0$ such that $\abs{\Spec_{1 - \eta} \mu} \leq \beta q$.  We observe that there is a value $\gamma > 0$ such that $\Spec_{1 - \gamma} \mu$ does not contain any non-trivial additive subgroups $H \leq \mathbb{F}_q^n$.  In fact, by Markov's inequality and Fourier inversion,
    \[
      (1 - \gamma)^2 \# H \cap \Spec_{1 - \gamma} \mu \leq \sum_{t \in H} \abs{\widehat{\mu}(t)}^2 \leq \abs{H} (1 - \alpha).
    \]
We then choose $\gamma = \alpha/2$.

We can now use Corollary~\ref{cor:iteratedkneser} to show that for any $k \geq 1$,
\[
\abss{\Spec_{1-\gamma} \mu \setminus \{0\} } \geq k \abs{\Spec_{1-k^{-2} \gamma} \mu \setminus \{0\}}
\]
so we pick $k = \beta^{-1}$ and let $\delta = \beta^2 \gamma = \beta^2 \alpha/2$. We then deduce that $\beta = \sqrt{c/(5\alpha)}$.
\end{proof}

\subsection{The Swapping Lemma}

Let $\mu$ be an $\alpha$-dense probability distribution and $V$ an unsaturated subspace of codimension $k$.  We want to find $\nu$ depending only on $\mu$ such that, if $Y$ is a random vector with iid entries taken from $\nu$, we have the inequality
\[
  \abs{\mathbb{P}(X \in V) - q^{-k}} \leq \left( \frac12 + o(1) \right) \abs{\mathbb{P}(Y \in V ) - q^{-k}}
\]
Let us postpone the definition of $\nu$ and define functions $f, g : \mathbb{F}_q \to \mathbb{R}^+$ to be
\begin{align*}
    f(t) &= \prod_{\ell = 1}^n \abs{\widehat{\mu}(w_\ell t)} \\
    g(t) &= \prod_{\ell = 1}^n \widehat{\nu}(w_\ell t).
\end{align*}
The lemma would follow immediately if we had $\widehat{\nu} \geq 0$ and we could establish
\[
  \sum_{t \in V^\perp \setminus \{0\}} f(t) \leq \left(\frac12 + o(1) \right) \sum_{t \in V^\perp \setminus \{0\}} g(t).
\]
Let $F(u) = \{t \mid f(t) \geq u\}$ and $G(u) = \{t \mid g(t) \geq u\}$ denote level sets.  We define $\nu$ so that the level sets $G(u)$ control the additive structure of $F(u)$.
\begin{proposition} \label{prop:constructmeasure}
There is a probability distribution $\nu : \mathbb{F}_q \to [0,1]$ depending on $\mu$ and $\alpha$ with the following properties.
\begin{enumerate}
    \item For all $0 < u < 1$ we have the sumset inclusion $F(u) + F(u) \subseteq G(u)$.
    \item For all $t \in V^\perp$, $f(t) \leq g(t)^4$.
    \item $\widehat{\nu}(t) \geq 0$ for all $t \in \widehat{\mathbb{F}}_q$.
    \item $\nu$ is $\beta$-dense for $\beta = \alpha/8$.
\end{enumerate}
\end{proposition}
We will prove Proposition~\ref{prop:constructmeasure} in a moment.  First we will show how to use Proposition~\ref{prop:constructmeasure} to prove Lemma~\ref{lem:replacement}.

\begin{proof}[Proof of Lemma~\ref{lem:replacement}]
Let $\eps > 0$ be determined later. We decompose the sum of $f$ into the domains where $f \leq \eps$ and $f > \eps$,
\[
\sum_{t \in V^\perp \setminus \{0\}} f(t) \leq \sum_{\substack{t \in V^\perp \setminus \{0\} \\ f(t) \leq \eps}} f(t) + \sum_{\substack{t \in V^\perp \setminus \{0\} \\ f(t) > \eps}} f(t).
\]
We can control the domain where $f \leq \eps$ using the inequality $f(t) \leq g(t)^4$.  Namely,
\[
\sum_{\substack{t \in V^\perp \setminus \{0\} \\ f(t) \leq \eps}} f(t) \leq \eps^{3/4} \sum_{t \in V^\perp \setminus \{0\}} g(t).
\]
Therefore if we can set $\eps = o(1)$ as $n \to \infty$ this part is complete.

We write the sum over the domain where $f > \eps$ into level sets,
\[
\sum_{\substack{t \in V^\perp \setminus \{0\} \\ f(t) > \eps}} f(t) = \int_\eps^\infty \abs{F'(u)} \, du + \eps \abs{F'(\eps)}
\]
Here we let $F'(u) := F(u) \setminus \{0\}$ and similarly define $G'(u) := G(u) \setminus \{0\}$.

From the sumset inequality $F(u) + F(u) \subseteq G(u)$ and Kneser's inequality,
\[
  2 \abs{F(u)} \leq \abs{\Sym (F(u) + F(u))} + \abs{G(u)}
\]
We would like to pick $\eps = o(1)$ such that $\abs{\Sym (F(u) + F(u))}$ = 1.
Since $F(u) + F(u) \subseteq G(u)$ and $G(u)$ is increasing in $u$, we require that every non-trivial additive subgroup $H \leq V^\perp$ contain a non-zero element $t \notin G(\eps)$.

Fix $H \leq V^\perp$.  We can clearly assume that $H \cong \Z/p\Z$; pick $w \in V^\perp$ that generates $H$.  Since $V$ is unsaturated, we know that $w$ contains at least $\delta n$ non-zero entries.

Define the function
\[
  h(t) := \sum_{\ell = 1}^n 1 - \widetilde{\nu}(t_\ell)^2.
\]
for $t \in H$.  Averaging $h$ over $H$, we can argue as in the proof of Lemma~\ref{lem:littlewoodofford} to find
\[
  \abs{H}^{-1} \sum_{t \in H} h(t) \geq \beta \delta n.
\]
Note that we need $\nu$ to be $\beta$-dense. By the pigeonhole principle we can find a (non-zero) $t \in H$ with $h(t) \geq \beta \delta n$. We then deduce that
\[
  g(t) \leq \exp(-\frac12 h(t)) \leq \exp(-\frac12 \beta \delta n)
\]
so we set $\eps = \exp(-\frac12 \beta \delta n)$.  For every $u \geq \eps$ we now have
\[
  2 \abs{F'(u)} \leq \abs{G'(u)},
\]
so returning to our integral of level sets we find
\[
\int_\eps^\infty \abs{F'(u)} \, du + \eps \abs{F'(\eps)} \leq \frac12 \int_0^\infty \abs{G'(u)} \, du.
\]
The lemma now follows.
\end{proof}

\begin{proof}[Proof of Proposition~\ref{prop:constructmeasure}]
    Let $\gamma = 1/8$ be a parameter and define
    \[
      \nu(t) := \begin{cases} \gamma \mu * \mu^-(t), & t \neq 0 \\ 1 - \sum_{s \neq 0} \nu(s), & t = 0. \end{cases}
    \]
    Clearly $\nu$ is a probability measure if $0 < \gamma < 1$. We also have $\widehat{\nu} > 1 - 2 \gamma$.  Let $\beta = \gamma\alpha$.  It is easy to see that $\nu$ is $\beta$-dense:  for $H \leq \mathbb{F}_q$ additive we have
\[
  \nu(H) = 1 - \sum_{t \notin H} \gamma \mu * \mu^-(t) \leq 1 - \gamma \alpha
\]
and for any $x \notin H$ we have
\[
  \nu(x + H) = \sum_{t \in x + H} \gamma \mu * \mu^-(t) \leq \gamma(1 - \alpha) \leq 1 - \gamma \alpha.
\]
as desired.

  The Fourier transform of $\nu$ is given by
\[
  \widehat{\nu}(\xi) = 1 - \gamma + \gamma \abs{\widehat{\mu}(\xi)}^2.
\]

We would next like to show that $F(u) + F(u) \subseteq G(u)$ for all $0 < u < 1$. It suffices to show, for all $\theta, \psi \in \widehat{\mathbb{F}}_q$,
\[
  \abs{\widehat{\mu}(\theta) \widehat{\mu}(\psi)} \leq \widehat{\nu}(\theta + \psi)^2.
\]
We will consider two cases.
    \begin{enumerate}
        \item Suppose $\abs{\widehat{\mu}(\theta)} < 1 - 4\gamma$ or $\abs{\widehat{\mu}(\psi)} < 1 - 4\gamma$.  Then
\[
  \abs{\widehat{\mu}(\theta) \widehat{\mu}(\psi)} < 1 - 4\gamma < (1 - 2\gamma)^2 < \widehat{\nu}^2(\theta + \psi).
\]
        \item Now suppose that $\abs{\widehat{\mu}(\theta)}, \abs{\widehat{\mu}(\psi)} \geq 1 - 4 \gamma$.  Define $\theta_1 = 1 - \abs{\widehat{\mu}(\theta)}$ and $\theta_2 = 1 - \abs{\widehat{\mu}(\psi)}$.  By Lemma~\ref{lem:spec}, we know that $\abs{\widehat{\mu}(\theta + \psi)}^2 \geq 1 - 2 (\theta_1 + \theta_2)$.
        
        We have the inequality
        \[
          \widehat{\nu}(\theta + \psi) = 1 - \gamma + \gamma \abs{\widehat{\mu}(\theta + \psi)}^2 \geq 1 - 4 \gamma(\theta_1 + \theta_2)
        \]
        Since we have $\gamma = 1/8$, we conclude that
        \[
          \widehat{\nu}(\theta + \psi)^2 \geq \abs{\widehat{\mu}(\theta) \widehat{\mu}(\psi)}
        \]
        as required.
    \end{enumerate}

    It remains to show that $\abs{\widehat{\mu}(\theta)} \leq \widehat{\nu}(\theta)^4$ for all $\theta$.  By the geometric-arithmetic mean inequality,
    \[
      (\abs{\widehat{\mu}(\theta)}^2 \cdot 1^7)^{1/8} \leq \frac18 (\abs{\widehat{\mu}(\theta)}^2 + 7) = \widehat{\nu}(\theta)
    \]
    as required.
\end{proof}


\section{Probability distribution of the determinant} \label{sec:probdist}

We will now indicate how to modify the proof of Theorem~\ref{thm:main} to prove Theorem~\ref{thm:probdistribution}.

Again let $X_1, \ldots, X_n$ denote the columns of $A$.  We begin by revealing all but the first column of the matrix.  If we abbreviate $W := \langle X_2, ..., X_n \rangle$ then we find
\[
  \mathbb{P}(\det M = t) = \mathbb{P}(\det M = t | \codim W = 1) \mathbb{P}(\codim W = 1)
\]
We now use Proposition~\ref{prop:uniformsubspaces} to control the last $n - 1$ vectors,
\begin{align*}
  \mathbb{P}(\codim W = 1) &= \prod_{k = 2}^n \mathbb{P}(X_k \notin \langle X_{k+1}, \ldots, X_n \rangle \mid \codim \langle X_{k+1}, \ldots, X_n \rangle = k) \\
  &= \prod_{k = 2}^\infty (1 - q^{-k}) + O(e^{-cn}).
\end{align*}

Pick $w \perp W$ such that $\det A = X_1 \cdot w$; namely, $w$ is the first row of the adjugate of $A$.

We can classify the possible hyperplanes $V$ that $W$ can represent. These are similar to the definitions made in Section~\ref{sec:maintheorem}, but the definition of semi-saturated has been expanded.
\begin{description}
  \item[sparse] We have $\abs{\supp w} \leq \delta n$. Note that this is well-defined independent of the choice of $w \perp V$.
  \item[unsaturated] $V$ is not sparse and either
  \[
    \max(e^{-d\alpha n}, Dq^{-1}) \leq \abs{\mathbb{P}(X \in V) - q^{-1}}
  \]
  or
  \[
    Dq^{-1} \leq \mathbb{P}(X \in V) \leq e^{-d\alpha n} \leq \mathbb{P}(X \cdot w = t)
  \]
  for some $t \in \mathbb{F}_q$.
  \item[semi-saturated] $V$ is not sparse, 
  \[
  \abs{\mathbb{P}(X \in V) - q^{-1}} < Dq^{-1}
  \]
  and there is a $t \in \mathbb{F}_q$ with
  \[
    e^{-d\alpha n} < \abs{\mathbb{P}(X \cdot w = t) - q^{-1}}.
  \]  We can control these with a modified the inverse theorem.
  \item[saturated] $V$ is not sparse and 
  \[
    \abs{\mathbb{P}(X \cdot w = t) - q^{-1}} \leq e^{-d\alpha n}
  \]
  for all $t \in \mathbb{F}_q$.
\end{description}

We will now show that $W$ represents sparse, semi-saturated, and unsaturated subspaces with probability $O(e^{-c\alpha n})$.

\subsection{Sparse subspaces}

The argument in Section~\ref{sec:sparsesubspaces} shows that these occur with probability $O(e^{-c\alpha n})$.

\subsection{Unsaturated subspaces}

Since $\mathbb{P}(X \cdot w = t) \leq \mathbb{P}(Y \in V)$, we see that regardless of which set of inequalities hold, we have
\[
  Dq^{-1} \leq \mathbb{P}(X \in V)
\]
and
\[
  e^{-d\alpha n} \leq \mathbb{P}(Y \in V).
\]
Therefore the argument from Section~\ref{sec:unsaturatedsubspaces} applies, so that unsaturated subspaces appear with probability $O(e^{-c\alpha n})$.

\subsection{Semi-saturated subspaces}

For all $t \in \mathbb{F}_p$ we can calculate
\[
  \mathbb{P}(X \cdot w = t) = q^{-1} \sum_{\xi \in \Z/(p)} e_p(-\Tr(t \xi)) \prod_{\ell = 1}^n \mathbb{E} e_p(\Tr(\psi_\ell w_\ell \xi))
\]
Rearranging and applying the triangle inequality,
\[
  \abs{\mathbb{P}(X \cdot w = t) - q^{-1}} \leq q^{-1} \sum_{\xi \in \mathbb{F}_q \setminus \{0\}} \prod_{\ell = 1}^n \abs{\cos(2 \pi w_\ell \xi)}
\]
The argument can now be completed as in Theorem~\ref{thm:main}. \qed

\section{Acknowledgments}

The author thanks Terence Tao for guidance and helpful conversation.

\bibliography{srmff}{}
\bibliographystyle{plain}

\end{document}